\documentclass[final,reqno]{amsart}

\usepackage{graphicx}
\usepackage{amsmath,amsthm,amsfonts,amscd,amssymb}
\usepackage[color,notref,notcite]{showkeys}
\usepackage{url}
\usepackage{subeqnarray}
\usepackage[color,notref,notcite]{showkeys} 
\definecolor{labelkey}{rgb}{0.6,0,1}
\usepackage{enumitem}
\usepackage{algorithm}
\usepackage{algpseudocode}
\usepackage{citesort}

\theoremstyle{plain}
\newtheorem{theorem}{Theorem}[section]

\newtheorem{lemma}[theorem]{Lemma}

\newtheorem{assumptions}[theorem]{Assumptions}

\theoremstyle{definition}
\newtheorem{definition}[theorem]{Definition}

\def\bhyp#1{\begin{equation}\label{#1}\begin{array}{c}}
\def\ehyp{\end{array}\end{equation}}

\newcounter{cst}

\theoremstyle{remark}
\newtheorem{remark}[theorem]{Remark}

\numberwithin{equation}{section}
\numberwithin{figure}{section}

\newcommand{\RR}{{\mathbb R}}
\newcommand{\NN}{{\mathbb N}}
\newcommand{\cB}{{\mathcal B}}
\newcommand{\cF}{{\mathcal F}}

\def\O{\Omega}
\def\dsp{\displaystyle}
\def\bfn{\mathbf{n}}

\def\disc{{\mathcal D}}
\def\mesh{{\mathcal M}}

\def\edge{\sigma}

\def\dr{\partial}

\def\bvarphi{\boldsymbol{\phi}}

\newcommand{\x}{\pmb{x}}

\def\cA{\mathcal{A}}

\def\bphi{\boldsymbol{\varphi}}
\def\bxi{\boldsymbol{\xi}}

\def\cE{\mathbb E}

\newif\ifcorr\corrtrue
\usepackage[normalem]{ulem}
\definecolor{violet}{rgb}{0.580,0.,0.827}

\newcommand{\ud}{\, \mathrm{d}} 
\def\div{\mathop{\rm div}}

\title[Analysis of schemes for convection-diffusion-reaction equations]{Error Estimates for Non Conforming Discretisation of Time-dependent Convection-Diffusion-Reaction Model}

\author{Hasan Alzubaidi}
\address[Hasan Alzubaidi]{Department of Mathematics, Al-Qunfudah University College, Umm Al-Qura University, Saudi Arabia}
\email{hmzubaidi@uqu.edu.sa}

\author{Yahya Alnashri}
\address[Yahya Alnashri]{Department of Mathematics, Al-Qunfudah University College, Umm Al-Qura University, Saudi Arabia}
\email{yanashri@uqu.edu.sa}


\date{\today}

\begin{document}

\begin{abstract}
We use a generic framework, namely the gradient discretisation method (GDM), to propose a unified numerical analysis for general  time-dependent convection-diffusion-reaction models. We establish novel results for convergence rates of numerical approximations of such models under reasonable assumptions on exact solutions, and prove the existence and uniqueness of the approximate solution for suitably small time steps. The main interest of our results lies in covering several approximation methods and various applications of the considered model such as the generalised Burgers-Fisher (GBF) and the generalised Burgers-Huxley (GBH) models.
Numerical tests based on the hybrid mimetic mixed (HMM) method for the GBF model are performed on various types of general meshes to examine the accuracy of the proposed gradient scheme. The results confirm our theoretical rates of convergence, even on mesh with extreme distortions.
\end{abstract}

\maketitle

\section{Introduction}
\label{sec-intro}
In this paper, we design and analyse an approximation of a solution to the time-dependent convection-diffusion-reaction model, which is of the following strong form:
\begin{equation}\label{eq-conv-1}
\begin{aligned}
\dr_t \bar c(\x,t)-\lambda\div(\nabla \bar c(\x,t))&+\cA(g(\bar c),\nabla\bar c)\\
&=f(\bar c), \quad (\x,t) \in \O \times (0,T),
\end{aligned}
\end{equation}
with the following homogeneous Dirichlet boundary and initial conditions:
\begin{equation}\label{eq-conv-2}
\bar c(\x,t)=0, \quad (\x,t) \in \dr\O \times(0,T),
\end{equation}
\begin{equation}\label{eq-conv-3}
\bar c(\x,0)=c_0, \quad \x \in \O.
\end{equation}

The model \eqref{eq-conv-1}--\eqref{eq-conv-3} covers many prototype mathematical models such as the generalised Burgers-Fisher (GBF) and the generalised Burgers-Huxley (GBH) models that describe various physical phenomena arising in different scientific fields including mathematical biology, plasma physics, fluid dynamics, financial mathematics, elasticity and heat conduction, transport phenomena, neuroscience and among others \cite{B3,B4}. In our numerical applications, we consider the generalised Burgers-Fisher model which has the properties of convective phenomenon from Burgers equation \cite{A1} and having diffusion transport as well as reactions kind of characteristics from Fisher equation \cite{A2}. 
\paragraph{}
 Generally, the lack of an analytical solution of the model \eqref{eq-conv-1}--\eqref{eq-conv-3} in its general setting leads researchers to seek numerical solutions that are accurate and reliable.  The existing literature has used a range of mathematical approaches to investigate and approximate the solutions of this type of model, especially its specific instances, such as the GBF and GBH models. 
 For example, with focusing on recent literature, Javidi in 2006 \cite{A3} studied spectral 
collocation method for the solution of the GBF equation. In 2008, Darvishi et al. \cite{A4} constructed a spectral collocation method from Chebyshev–Gauss–Lobatto collocation points to examine the GBH model.  Golbabai and Javidi in 2009, used a spectral domain decomposition approach based on Chebyshev polynomials to solve the GBF model \cite{A5}. 
  Sari et al. \cite{A6} solved the GBF problem in 2010 using a compact finite difference method that requires very low computing power. In 2011, the GBH model has been analysed using a three-step    Taylor–Galerkin finite element scheme \cite{A7} and an efficient finite difference approach \cite{A8}.  Tatari et al. \cite{A9} in 2012, implemented a collocation method based on radial basis for solving the GBF equation. Ervin et al. \cite{A10}  presented numerical solutions based on finite element methods in 2015 that can provide bounded and non-negative solutions to the GBH equation. 
   The GBH model was solved using a Chebyshev wavelet collocation technique in 2016 by Çelik \cite{A11}, and Chandraker et al. \cite{A12} proposed two implicit finite difference schemes to solve the GBF equation. 
Macías-Díaz and Gonzalez presented an exact finite-difference technique in 2017 \cite{A13} to obtain the bounded and positive solutions of the classical Burgers-Fisher equation.
 Moreover, the classical BF problem was approximated and solved using finite elements by Yadav and Jiwari  \cite{A14}, who also achieved a priori error estimates and convergence of semi-discrete solutions. 
 The GBF model was solved by Namjoo et al. in 2018 using a non-standard finite-difference method, and the positivity, consistency, and boundedness of the scheme were further discussed  \cite{A15}. 
In 2019, Alinia and Zarebnia \cite{A16} used a numerical approach with a tension hyperbolic-trigonometric B-spline scheme for solving the GBH model. 
Followed by Hussain and Haq \cite{A17} in 2020, who proposed  a meshfree spectral interpolation technique combined with Crank–Nicolson difference scheme to solve the GBF equation. 
With a discontinuous, nonconforming, and conforming Galerkin finite element method, the stationary GBH was examined in 2021 \cite{A18}.  
 Additionally, using a Faedo-Galerkin approximation approach, Mohan and Khan \cite{A19} established the existence and uniqueness of a global weak solution of the GBH model. 
 
  An improvised collocation method using cubic B-splines was employed by Shallu and Kukreja \cite{A20} in 2022 to obtain accurate solutions for the GBH model.
  Chin \cite{A21} later developed an effective numerical method in 2023, which utilized a combination of the Galerkin method in the space variables and the non-standard finite difference method in the time variables.
    
  Recently, the solutions of the classical BF model have been analysed  and approximated using a modified version of the finite element method known as virtual element method in 2024 \cite{A22}, and using a three-level linearised finite difference schemes in 2025 \cite{A23}.  
\paragraph{}
It is worth mentioning that the analysis discussed in the existence literature, where the above list is just a sample,  focused on the conforming numerical techniques of the model \eqref{eq-conv-1}--\eqref{eq-conv-3}, or on its particular cases such as the GBF and GBH equations. This motivates us to present a generic nonconforming discretisation of the studied model in its general form using  a generic framework analysis known as a gradient discretisation method (GDM). The GDM is an abstract framework to construct a unified convergence analysis of numerical schemes for different types of partial differential equations. It covers a variety of conforming and non conforming numerical methods, for instance, conforming, non-conforming and mixed finite elements methods, hybrid mimetic mixed methods, SUSHI scheme, mixed finite volumes, nodal mimetic finite differences, and multi-points flux approximation method. We refer the reader to \cite{B1} for more details.
\paragraph{} 
Key contributions of the current work are as follows:  
\begin{itemize}
\item \textit{We provide general error estimates for generic approximations of the model \eqref{eq-conv-1}--\eqref{eq-conv-3} using the GDM. To the best of our knowledge, these results appear to be novel, particularly in the context of nonconforming discretizations. Unlike previous studies, which often focus on specific forms of convection and reaction terms and use specific numerical methods, our analysis addresses the model in its general form, and can be applied to a wide range of schemes that fit within the GDM framework.}
\item \textit{The GBF and GBH equations, which arise in numerous physical and biological applications, fall within the scope of the time-dependent convection-diffusion-reaction model \eqref{eq-conv-1}--\eqref{eq-conv-3}. While earlier works typically examine such models in isolation, our approach provides a general numerical treatment applicable across different variants and applications, including a focused numerical study on the GBF model.} 
\item \textit{ The GDM framework analysis which is used to design a complete numerical analysis for the studied model, covers many conforming and non-conforming classical methods. Up to the best of our knowledge, the previous works only dealt with conforming methods that are unable to maintain the physical properties for the studied model, particularly on some kind of general meshes.
In contrast, our implementation of the hybrid mimetic mixed (HMM) method, a finite volume technique, demonstrates robust performance. We validate its effectiveness through numerical experiments on four different generic mesh types, including  meshes with extreme distortions.}
\end{itemize}

The organisation of this paper is structured as follows.  Section \ref{sec-weak} is devoted to the weak formulation of the problem and the approximation method. In Section \ref{sec-theorem}, we state and prove the main results, the error estimates. We develop a new technique to deal with general non linear reaction and convection terms. Section \ref{sec-numerical} contains some numerical experiments using the HMM method for a prototype example of the studied model known as the generalised Burgers-Fisher (GBF) model. The tests are performed on four different general meshes and the resultant relative errors with respect to the mesh size are reported.


\section{Variational formulation and an approximate scheme}\label{sec-weak}

\begin{assumptions}\label{assump-rm}
The assumptions on the model data are the following:
\begin{itemize}
\item the domain $\O$ is an open bounded connected subset of $\RR^d\; (d > 1)$ with a boundary $\dr\O$, and $T>0$,
\item the diffusion coefficient $\lambda >0$,
\item the non linear operator $\cA(\cdot,\cdot):\RR \times \RR^d \to \RR$ is a Lipschitz continuous with a positive constants $\ell_1$, i.e. for all $\psi_1,\psi_2 \in L^2(0,T;L^2(\O))$ and $\bphi_1,\bphi_2 \in L^2(0,T;L^2(\O)^d)$, the following holds:
\[
\begin{aligned}
\|\cA(\psi_1,\bphi_1)&-\cA(\psi_2,\bphi_2)\|_{L^2(0,T;L^2(\O)) \times L^2(0,T;L^2(\O)^d)}\\
&\leq \ell_1\Big( \|\psi_1 - \psi_2\|_{L^2(0,T;L^2(\O))}
+\|\bphi_1 - \bphi_2\|_{L^2(0,T;L^2(\O)^d)} \Big),
\end{aligned}
\] 
\item the non linear convection and reaction functions $g$ and $f$ are Lipschitz continuous with positive constants $\ell_2$ and $\ell_3$, respectively
\item the initial solution $c_0 \in L^2(\O)$.
\end{itemize}
\end{assumptions}

In what follows, we denote by $\langle \cdot,\cdot \rangle_H$ and $\langle \cdot,\cdot \rangle$ the duality product between $H^1(\O)$ and $H^{-1}(\O)$ and the the scalar product in $L^2(\O)$ or in $L^2(\O)^d$, respectively. 

Let Assumptions \ref{assump-rm} hold. The problem \eqref{eq-conv-1}--\eqref{eq-conv-3} admits a weak solution if we find a function $\bar c \in L^2(0,T;H_0^1(\O))$ such that $\dr_t\in L^2(0,T;H^{-1}(\O))$ and $\bar c(0)=c_0$, and for a.e. $t\in [0,T]$, the following equality is satisfied:
\begin{equation}\label{eq-weak-model}
\begin{aligned}
\langle \dr_t \bar c,\varphi \rangle_{H}
+\lambda \langle \nabla\bar c(t),\nabla \varphi \rangle
&+\langle\cA(g(\bar c(t)),\nabla\bar c(t)),\varphi \rangle \\
&=\langle f(\bar c(t)),\varphi \rangle,\quad \forall \varphi \in H_0^1(\O).
\end{aligned}
\end{equation}


Using the Gradient Discretisation Method, we present a general approximation scheme for the problem \eqref{eq-weak-model}, substituting the continuous operators with their discrete equivalent.

\begin{definition}\label{def-gd}
A gradient discretisation for the problem \eqref{eq-weak-model} is defined by $\disc=(X_{\disc,0}, \Pi_\disc, \nabla_\disc,J_\disc, (t^{(m)})_{0\leq m\leq N} )$, where
\begin{enumerate}
\item the set of discrete unknowns $X_{\disc,0}$ is a finite dimensional space on $\RR$, corresponding together to the interior unknowns and to the boundary unknowns,
\item the function reconstruction $\Pi_\disc : X_{\disc,0} \to L^2(\O)$ is a linear,
\item the gradient reconstruction $\nabla_\disc : X_{\disc,0} \to L^2(\O)^d$ is a linear and must be defined so that $|| \nabla_\disc \cdot ||_{L^2(\O)^d}$ defines a norm on $X_{\disc,0}$,
\item $J_\disc:L^\infty(\O) \to X_{\disc,0}$ is a linear and continuous interpolation operator for the initial conditions,
\item $0=t^{(0)}<t^{(1)}<...<t^{(N)}=T$ are time steps.
\end{enumerate}
\end{definition}

\begin{itemize}
\item The stability of the generic gradient discretisation is assessed by the constant $C_\disc$ defined by
\begin{equation}\label{eq-coerc}
C_\disc = \max_{u\in X_{\disc,0}-\{0\}}\dsp\frac{\|\Pi_\disc u \|_{L^2(\O)}}{\| \nabla_\disc u \|_{L^2(\O)^d}},
\end{equation}
which leads to the discrete Poincar\'e inequality
\begin{equation}\label{ponc-enq}
\|\Pi_\disc u\|_{L^2(\O)} \leq C_\disc \|\nabla_\disc u\|_{L^2(\O)^d}.
\end{equation}

\item The consistency of the generic gradient discretisation is assessed by the function $S_\disc:H_0^1(\O) \to X_{\disc,0}$ defined by
\begin{equation}\label{eq-consist}
\begin{aligned}
S_\disc(w)= \min_{u\in X_\disc}\left(
\| \Pi_\disc u - w \|_{L^{2}(\O)}
+ \| \nabla_\disc u - \nabla w \|_{L^{2}(\O)^{d}}
\right).
\end{aligned}
\end{equation}

\item The limit-conformity of the generic gradient discretisation is assessed by the function $S_\disc:H_{\rm div} \to [0,\infty)$ defined by
\begin{equation}\label{eq-conformity}
\begin{aligned}
W_\disc(\bxi) = \max_{u\in X_{\disc,0}-\{0\}}
\dsp\frac{\left|
\langle \nabla_\disc u,\bxi \rangle
+ \langle \Pi_\disc u,\div \bxi \rangle
\right|}
{\| \nabla_\disc u \|_{L^{2}(\O)^d}},
\end{aligned}
\end{equation}
where $H_{\rm div}:=\{\bxi \in L^2(\O)^d\;:\; {\rm div}\bxi \in L^2(\O),\; \bxi\cdot\bfn=0 \mbox{ on } \dr\O  \}$.
\end{itemize}

\begin{definition}[an Implicit Gradient Scheme]
Let $\delta t^{(m+\frac{1}{2})}=\max_{m=0,..,N} t^{(m)}$. We define the discrete operator $\dr_\disc$ to approximate the derivative in time by
\[
\dr_\disc u^{(m+1)}=\dsp\frac{u^{(m+1)}-u^{(m)}}{\delta t^{(m+\frac{1}{2})}}.
\]
If $\disc$ is a gradient discretisation, we define the approximate scheme for the problem \eqref{eq-weak-model} by: Find $c:=(c^{(m)})_{\in\NN} \in X_{\disc,0}^{N+1}$, such that 
\begin{itemize}
\item $J_\disc(c_0)=c^{(0)}$,
\item for any $m\in [0,N]$,
\begin{equation}\label{eq-gs}
\begin{aligned}
\langle \dr_\disc c^{(m+1)},\Pi_\disc \varphi \rangle
+\lambda \langle \nabla_\disc c^{(m+1)},\nabla_\disc \varphi \rangle
&+\langle \cA(g(\Pi_\disc c^{(m+1)}),\nabla_\disc c^{(m+1)}),\Pi_\disc \varphi \rangle\\
&=\langle  f(c^{(m+1)}),\Pi_\disc \varphi \rangle,
\quad \forall \varphi \in X_{\disc,0}.
\end{aligned}
\end{equation} 
\end{itemize}
\end{definition}

\begin{lemma}
Let Assumptions \ref{assump-rm} hold, and $\disc$ be a gradient discretisation. If $\delta t^{(m+\frac{1}{2})}<\frac{2\lambda}{C_\disc+\varepsilon}$, such that $\varepsilon>0$, then the approximate scheme \eqref{eq-gs} admits a unique solution.
\end{lemma}

\begin{proof}
Assume that $c^{(m)}$ is known and unique in the scheme \eqref{eq-gs}. This means that we solve a square system of non-linear elliptic equations with variables $c^{(m+1)}$ at each iteration $m+1$. We use here the Brouwer’s fixed point theorem. If we fix $u\in X_{\disc,0}$, we can see that there exists a unique function $c\in X_{\disc,0}$ satisfying the following linear square system:
\begin{equation}\label{ex-u-1}
\begin{aligned}
\frac{1}{\delta t^{(m+\frac{1}{2})}}\langle \Pi_\disc(c-c^{(m)}),\Pi_\disc \varphi \rangle
&+\lambda \langle \nabla_\disc c,\nabla_\disc \varphi \rangle
+\langle \cA(g(\Pi_\disc u),\nabla_\disc c),\Pi_\disc \varphi \rangle\\
&\qquad=\langle f(\Pi_\disc u),\varphi \rangle, \quad \forall \varphi \in X_{\disc,0}.
\end{aligned}
\end{equation}
Let $\cB:X_{\disc,0}\to X_{\disc,0}$ be a mapping such that $\cB(u)=c$ in which $c$ solves the above problem. To establish the existence and uniqueness of the discrete solution $c$, it is enough to show that $\cB$ is a contractive mapping with respect to the normed space $X_{\disc,0}$. For $u,\tilde u \in X_{\disc,0}$ such that $c=\cB(u)$ and $\tilde c=\cB(\tilde u)$, we have 
\begin{equation}\label{ex-u-1-a}
\begin{aligned}
\frac{1}{\delta t^{(m+\frac{1}{2})}}\langle \Pi_\disc(c-c^{(m)}),\Pi_\disc \varphi \rangle
&+\lambda \langle \nabla_\disc c, \nabla_\disc \varphi \rangle
+\langle \cA(g(\Pi_\disc u),\nabla_\disc c),\Pi_\disc \varphi \rangle\\
&\qquad=\langle f(\Pi_\disc u),\Pi_\disc \varphi \rangle, \quad \forall \varphi \in X_{\disc,0},
\end{aligned}
\end{equation}
\begin{equation}\label{ex-u-1-b}
\begin{aligned}
\frac{1}{\delta t^{(m+\frac{1}{2})}}\langle \Pi_\disc(\tilde c-c^{(m)}),\Pi_\disc \varphi \rangle
&+\lambda \langle \nabla_\disc \tilde c, \nabla_\disc \varphi \rangle
+\langle \cA(g(\Pi_\disc \tilde u),\nabla_\disc \tilde c),\Pi_\disc \varphi \rangle\\
&\qquad=\langle f(\Pi_\disc \tilde u),\Pi_\disc \varphi \rangle, \quad \forall \varphi \in X_{\disc,0},
\end{aligned}
\end{equation}
Subtracting \eqref{ex-u-1-b} from \eqref{ex-u-1-a} yields, for all $\varphi,\; \psi \in X_{\disc,0}$, 
\begin{equation}\label{ex-u-2}
\begin{aligned}
\frac{1}{\delta t^{(m+\frac{1}{2})}}\langle \Pi_\disc(c&-\tilde c),\Pi_\disc \varphi \rangle
+\lambda \langle \nabla_\disc (c-\tilde c), \nabla_\disc \varphi \rangle\\
&+\langle \cA(g(\Pi_\disc u),\nabla_\disc c)-\cA(g(\Pi_\disc \tilde u),\nabla_\disc \tilde c),\Pi_\disc \varphi \rangle\\
&\quad=\langle f(\Pi_\disc u)
-f(\Pi_\disc \tilde u),\Pi_\disc \rangle.
\end{aligned}
\end{equation}
Taking $\varphi=\frac{1}{\delta t^{(m+\frac{1}{2})}}(c-\tilde c)$ in the above equation, one has
\begin{equation}\label{ex-u-3}
\begin{aligned}
&\frac{1}{(\delta t^{(m+\frac{1}{2})})^2}\|\Pi_\disc(c-\tilde c)\|_{L^2(\O)}^2
+\frac{\lambda}{\delta t^{(m+\frac{1}{2})}}\|\nabla_\disc(c-\tilde c)\|_{L^2(\O)^d}^2\\
&\leq \frac{1}{\delta t^{(m+\frac{1}{2})}}\langle f(\Pi_\disc u)
-f(\Pi_\disc \tilde u),\Pi_\disc(c-\tilde c) \rangle\\
&\quad-\frac{1}{\delta t^{(m+\frac{1}{2})}}\langle \cA(g(\Pi_\disc u),\nabla_\disc c)-\cA(g(\Pi_\disc \tilde u),\nabla_\disc \tilde c),\Pi_\disc (c-\tilde c) \rangle.
\end{aligned}
\end{equation}
To estimate the first term on the right-hand side of the above inequality, we first employ the Cauchy–Schwarz inequality, followed by Young's inequality. We obtain, thanks to the Lipschitz continuity conditions on $F$
\begin{equation}\label{ex-u-4}
\begin{aligned}
&\frac{1}{\delta t^{(m+\frac{1}{2})}}\langle f(\Pi_\disc u)
-f(\Pi_\disc \tilde u),\Pi_\disc(c-\tilde c) \rangle \\
&\leq L\|\Pi_\disc u - \Pi_\disc \tilde u\|_{L^2(\O)} \|\frac{1}{\delta t^{(m+\frac{1}{2})}}\Pi_\disc(c-\tilde c)\|_{L^2(\O)}\\
&\leq \frac{L^2}{2}\|\Pi_\disc u - \Pi_\disc \tilde u\|_{L^2(\O)}^2 
+\frac{1}{2}\|\frac{1}{\delta t^{(m+\frac{1}{2})}}\Pi_\disc(c-\tilde c)\|_{L^2(\O)}^2,
\end{aligned}
\end{equation}
where $L:=\max\{\ell_1,\ell_2,\ell_3\}$. Now, to establish a bound on the second term on the right-hand side of the inequality \eqref{ex-u-3}, we apply the Cauchy–Schwarz inequality. We obtain, thanks to the Lipschitz continuity conditions on the operator $\cA$ and the function $g$
\begin{equation*}
\begin{aligned}
&\frac{1}{\delta t^{(m+\frac{1}{2})}}\langle \cA(g(\Pi_\disc u),\nabla_\disc c)-\cA(g(\Pi_\disc \tilde u),\nabla_\disc \tilde c),\Pi_\disc (c-\tilde c) \rangle \\
&\leq L\|(g(\Pi_\disc u),\nabla_\disc c)-(g(\Pi_\disc \tilde u),\nabla_\disc \tilde c)\|_{L^2(\O)\times L^2(\O)^d}
\|\frac{1}{\delta t^{(m+\frac{1}{2})}}\Pi_\disc (c-\tilde c)\|_{L^2(\O)}\\
&\leq \Big(L\|g(\Pi_\disc u)-g(\Pi_\disc \tilde u)\|_{L^2(\O)} + \ell_1\|\nabla_\disc (c-\tilde c)\|_{L^2(\O)^d} \Big) \|\frac{1}{\delta t^{(m+\frac{1}{2})}}\Pi_\disc (c-\tilde c)\|_{L^2(\O)^d} \\
& \leq \Big(L^2 \|\Pi_\disc (u-\tilde u)\|_{L^2(\O)} + L\|\nabla_\disc (c-\tilde c)\|_{L^2(\O)^d} \Big) \|\frac{1}{\delta t^{(m+\frac{1}{2})}}\Pi_\disc (c-\tilde c)\|_{L^2(\O)}\\
&\leq L^2\|\Pi_\disc (u-\tilde u)\|_{L^2(\O)} \|\frac{1}{\delta t^{(m+\frac{1}{2})}}\Pi_\disc (c-\tilde c)\|_{L^2(\O)^d}\\
&\quad+ L\|\nabla_\disc (c-\tilde c)\|_{L^2(\O)^d}
\|\frac{1}{\delta t^{(m+\frac{1}{2})}}\Pi_\disc (c-\tilde c)\|_{L^2(\O)},
\end{aligned}
\end{equation*}
which leads to, with the application of Young's inequality (with a small parameter $\varepsilon >0$) 
\begin{equation}\label{ex-u-4-new}
\begin{aligned}
&\frac{1}{\delta t^{(m+\frac{1}{2})}}\langle \cA(g(\Pi_\disc u),\nabla_\disc c)-\cA(g(\Pi_\disc \tilde u),\nabla_\disc \tilde c),\Pi_\disc (c-\tilde c) \rangle \\
&\leq \frac{1}{2} \|\Pi_\disc (u-\tilde u)\|_{L^2(\O)}^2 
+\frac{L^4\varepsilon+L^2}{2\varepsilon}\|\frac{1}{\delta t^{(m+\frac{1}{2})}}\Pi_\disc (c-\tilde c)\|_{L^2(\O)^d}^2\\
&\quad+\frac{\varepsilon}{2}\|\nabla_\disc (c-\tilde c)\|_{L^2(\O)^d}^2.
\end{aligned}
\end{equation}
Plugging \eqref{ex-u-4} and \eqref{ex-u-4-new} into \eqref{ex-u-3}, we obtain
\begin{equation*}
\begin{aligned}
(\frac{\lambda}{\delta t^{(m+\frac{1}{2})}}-\frac{\varepsilon}{2})\|\nabla_\disc(c-\tilde c)\|_{L^2(\O)^d}^2
\leq \|\Pi_\disc u - \Pi_\disc \tilde u\|_{L^2(\O)}^2.
\end{aligned}
\end{equation*}
Owing to the discrete Poincar\'e inequality \eqref{ponc-enq}, we conclude
\begin{equation*}
\begin{aligned}
\frac{2\lambda-\varepsilon \delta t^{(m+\frac{1}{2})}}{2\delta t^{(m+\frac{1}{2})}}\|\nabla_\disc(c-\tilde c)\|_{L^2(\O)^d}^2
\leq C_\disc\|\nabla_\disc (u - \tilde u)\|_{L^2(\O)^d}^2.
\end{aligned}
\end{equation*}
Take the square root of both sides to get
\begin{equation*}
\begin{aligned}
\|\nabla_\disc(c&-\tilde c)\|_{L^2(\O)^d}
\leq \sqrt{\frac{C_\disc \delta t^{(m+\frac{1}{2})}}{2\lambda-\varepsilon \delta t^{(m+\frac{1}{2})}} }\|\nabla_\disc (u -\tilde u)\|_{L^2(\O)^d}.
\end{aligned}
\end{equation*}
Since $c=\cB(u)$ and $\tilde c=\cB(\tilde u)$, we arrive at
\begin{equation*}\label{ex-10}
\begin{aligned}
\|\cB(u)-\cB(\tilde u)\|_{X_{\disc,0}}
\leq C_4 \|u-\tilde u\|_{X_{\disc,0}},
\end{aligned}
\end{equation*}
which proves that the mapping $\cB$ is contractive under the condition $\delta t^{(m+\frac{1}{2})}<\frac{2\lambda}{C_\disc+\varepsilon}$, such that $\varepsilon>0$, and it has a unique fixed point $\cB(u)=c$.
\end{proof}


\section{Error Estimates}\label{sec-theorem}
\begin{theorem}\label{thm-err-rm} 
Under Hypothesis \ref{assump-rm} and the condition that $\delta t^{(m+\frac{1}{2})}<\frac{2\lambda}{C_\disc+\varepsilon}$, such that $\varepsilon>0$, let $\disc$ be a gradient discretisation, and $\bar c$ and $c$ be the solutions to the continuous problem \eqref{eq-weak-model} and to the approximate scheme \eqref{eq-gs}, respectively. If $\bar c \in W^{1,\infty}(0,T;W^{2,\infty}(\O))$, then the following error estimates hold:
\begin{subequations}\label{eq-error}
\begin{equation}\label{eq-error-1}
\begin{aligned}
\Big\| \Pi_\disc c(\cdot,t)&-\bar c(\cdot,t) \Big\|_{L^\infty(0,T;L^2(\O))}\\
&\leq C_1\Big[\delta t 
+S_\disc(\bar c(0))
+\|c_{\rm ini}-\Pi_\disc J_\disc c_{\rm ini} \|_{L^2(\O)}\\
&\quad+\dsp\sum_{m=0}^{N-1}\delta t^{(m+\frac{1}{2})}\mathbb M_\disc^{(m+1)}\Big]
+\sqrt{2}S_\disc(\bar c(t^{(k)})),\quad \forall k\in\{1,...,N\},
\end{aligned}
\end{equation}
\begin{equation}\label{eq-error-2}
\begin{aligned}
\Big\| \nabla_\disc c&-\nabla\bar c \Big\|_{L^2(\O\times(0,T))^d}\\
&\leq C_2\Big[ \delta t 
+S_\disc(\bar c(0))
+\|c_{\rm ini}-\Pi_\disc J_\disc c_{\rm ini} \|_{L^2(\O)}\\
&\quad+\dsp\sum_{m=0}^{N-1}\delta t^{(m+\frac{1}{2})}\mathbb M_\disc^{(m+1)}\Big]
+\sqrt{2}\dsp\sum_{m=0}^{N-1}\delta t^{(m+\frac{1}{2})}S_\disc(\bar c(t^{(m+1)})),
\end{aligned}
\end{equation}
\end{subequations}
where $C_i>0$ for all $i\in \NN$ does not depend on discrete data, and
\begin{equation}\label{eq-MD}
\mathbb M_\disc^{(m+1)}:=\dsp\sum_{m=0}^{k-1} \delta t^{(m+\frac{1}{2})}\Big[ \delta t + S_\disc(\bar c(t^{(m+1)})) + S_\disc(\dr_t\bar c^{(m+1)}) + W_\disc(\nabla\bar c^{(m+1)}) \Big]^2.
\end{equation}
\end{theorem}

\begin{proof}
Let us begin by defining the operator $P_\disc H_0^1(\O)\to X_{\disc,0}$ by
\begin{equation}\label{PD}
P_\disc(w):=\arg\min_{u\in X_{\disc,0}} \|\Pi_\disc u - w\|_{L^2(\O)}
+\|\nabla_\disc u - \nabla w\|_{L^2(\O)}.
\end{equation}
It is proved in \cite[Lemma 1]{B1} that it is linear and satisfies the following property.
\begin{equation}\label{PD-r}
S_\disc(w)= \Big( \|\Pi_\disc P_\disc w - w\|_{L^2(\O)}
+\|\nabla_\disc P_\disc w - \nabla w\|_{L^2(\O)} \Big)^\frac{1}{2}, \quad \forall w\in H_0^1(\O).
\end{equation}
Additionally, based on the definition of $S_\disc$ in \eqref{eq-consist}, it follows that, for any $w\in H_0^1(\O)$,
\begin{equation}\label{new-eq-proof-100}
\begin{aligned}
\Big( \|\Pi_\disc P_\disc w &- w\|_{L^2(\O)}^2
+\|\nabla_\disc P_\disc w - \nabla w\|_{L^2(\O)^d}^2  \Big)^{1/2}
\\
&\leq \Big( \|\Pi_\disc u - w\|_{L^2(\O)}^2
+\|\nabla_\disc u - \nabla w\|_{L^2(\O)^d}^2  \Big)^{1/2}, \quad \forall u \in X_{\disc,0}\\
&\leq \sqrt{2}S_\disc(w),\quad \forall u \in X_{\disc,0}.
\end{aligned}
\end{equation}
Given that $\nabla\bar c$ is a Lipschitz-continuous with respect to time, using the \eqref{PD} with $w:=\bar c(t^{(m+1)})$ yields
\begin{equation}\label{eq-proof-1-u}
\begin{aligned}
\Big\| \nabla\bar c^{(m+1)}&-\nabla_\disc  P_\disc \bar c(t^{(m+1)}) \Big\|_{L^2(\O)^d}\\
&\leq \Big\| \nabla\bar c^{(m+1)}-\nabla\bar c(t^{(m+1)}) \Big\|_{L^2(\O)^d} + S_\disc(\bar c(t^{(m+1)}))\\
&\leq C_3\delta t+S_\disc(\bar c(t^{(m+1)})).
\end{aligned}
\end{equation}
We recognise that $\dr_t \bar c^{(m+1)}$ belongs to ${H^2(\O)}$. By utilising the linear interpolant $P_\disc$ with $w:=\dr_t\bar c^{(m+1)}=\frac{\bar c(t^{(m+1)})-\bar c(t^{(m)})}{\delta t^{(m+\frac{1}{2})}}$, we obtain
\begin{equation}\label{eq-proof-2-u}
\begin{aligned}
\Big\| \frac{\Pi_\disc  P_\disc\bar c(t^{(m+1)})-\Pi_\disc  P_\disc \bar c(t^{(m)})}
{\delta t^{(m+\frac{1}{2})}}
-\partial_t \bar c^{(m+1)} \Big\|_{L^2(\O)}
&\leq S_\disc(\partial_t\bar c^{(m+1)}).
\end{aligned}
\end{equation}
Now, we can employ the definition \eqref{eq-conformity} of $W_\disc$ to $\bxi=:\nabla\bar c^{(m+1)} \in H_{\div}(\O)$ to deduce, for all $u\in X_{\disc,0}$,
\begin{equation}\label{eq-limi-conf-u}
\begin{aligned}
\langle \Pi_\disc u,\div(\nabla\bar c^{(m+1)} \rangle
&+\langle \nabla\bar c^{(m+1)},\nabla_\disc u \rangle \\
&\leq W_\disc(\nabla\bar c^{(m+1)}) \| \nabla_\disc u \|_{L^2(\O)^d}.
\end{aligned}
\end{equation}
Since $\dr_t\bar c^{(m+1)} + \cA(g(\bar c^{(m+1)}),\nabla\bar c^{(m+1)})- f(\bar c^{(m+1)})=\lambda\div(\nabla\bar c^{(m+1)})$ holds a.e. in space and time, the above inequality gives
\begin{equation}\label{eq-proof-3-u}
\begin{aligned}
\langle \Pi_\disc u,\dr_t\bar c^{(m+1)} &+\cA(g(\bar c^{(m+1)}),\nabla\bar c^{(m+1)})- f(\bar c^{(m+1)}) \rangle
+\lambda \langle \nabla\bar c^{(m+1)}, \nabla_\disc u \rangle \\
&\leq W_\disc(\nabla\bar c^{(m+1)}) \| \nabla_\disc u \|_{L^2(\O)^d},\quad \forall u\in X_{\disc,0}.
\end{aligned}
\end{equation}
By adding the terms $\pm f(\Pi_\disc c^{(m+1)})$ to the above inequality, we have
\begin{equation}\label{eq-proof-4-u}
\begin{aligned}
\langle\Pi_\disc u, \dr_t\bar c^{(m+1)}&- f(\Pi_\disc c^{(m+1)})\rangle
+\langle \Pi_\disc u, f(\Pi_\disc c^{(m+1)}) - f(\bar c^{(m+1)}) \rangle\\
&\quad+\lambda \langle \nabla\bar c^{(m+1)},\nabla_\disc u \rangle \\
&\leq W_\disc(\nabla\bar c^{(m+1)}) \| \nabla_\disc u \|_{L^2(\O)^d},\quad \forall u\in X_{\disc,0}.
\end{aligned}
\end{equation}
We note that $\|\nabla\bar c^{(m+1)}\|_{W^{1,\infty}(\O)^d}$ is bounded. Using the fact that $c$ is the solution to the approximate scheme \eqref{eq-gs}, we achieve 
\begin{equation}\label{eq-proof-5-u}
\begin{aligned}
\langle \dr_t\bar c^{(m+1)}&- \delta_\disc^{(m+\frac{1}{2})}c, \Pi_\disc u \rangle
+\lambda \langle \nabla\bar c^{(m+1)}-\nabla_\disc c^{(m+1)},\nabla_\disc u \rangle \\
&\leq \langle \Pi_\disc u, f(\bar c^{(m+1)})-f(\Pi_\disc c^{(m+1)}\rangle\\
&\quad+\langle \cA(g(\Pi_\disc c^{(m+1)}),\nabla_\disc c^{(m+1)})-\cA(g(\bar c^{(m+1)}),\nabla\bar c^{(m+1)}), \Pi_\disc u\rangle\\
&\quad+W_\disc(\nabla\bar c^{(m+1)}) \| \nabla_\disc u \|_{L^2(\O)^d},\quad \forall u\in X_{\disc,0}.
\end{aligned}
\end{equation}
For $n=1,...,N$, consider the notions $\cE^{(n)}:=P_\disc\bar c(t^{(n)})-c^{(n)}$. It follows that
\[
\begin{aligned}
\delta_\disc^{(m+\frac{1}{2})}\cE:&=\frac{ \Pi_\disc \cE^{(n+1)}-\Pi_\disc \cE^{(m)} }{ \delta t^{(m+\frac{1}{2})} }\\
&=\Big(  \dsp\frac{ \Pi_\disc  P_\disc\bar c(t^{(m+1)})-\Pi_\disc  P_\disc(\bar c(t^{(m)})) }{ \delta t^{(m+\frac{1}{2})} }-\dr_t\bar c^{(m+1)} \Big)+\Big( \dr_t\bar c^{(m+1)}-\delta_\disc^{(m+\frac{1}{2})} c \Big),
\end{aligned}
\]
and
\[
\nabla_\disc \cE^{(m+1)}= \left( \nabla_\disc  (P_\disc\bar c(t^{(m+1)}))-\nabla\bar c^{(m+1)}  \right) + \left(\nabla\bar c^{(m+1)}-\nabla_\disc c^{(m+1)}  \right).
\]
Together with \eqref{eq-proof-5-u}, \eqref{eq-proof-2-u} and \eqref{eq-proof-1-u}, and employing the Cauchy-Schwarz inequality it yields, for all $u\in X_{\disc,0}$,
\begin{equation}\label{eq-proof-6-u}
\begin{aligned}
&\langle \delta_\disc^{(m+\frac{1}{2})}\cE,\Pi_\disc u \rangle
+\langle \nabla_\disc \cE^{(m+1)}, \nabla_\disc u \rangle\\
&\quad\leq \langle f(\bar c^{(m+1)})-f(\Pi_\disc c^{(m+1)}), \Pi_\disc u \rangle\\
&\quad\quad+\langle \cA(g(\Pi_\disc c^{(m+1)}),\nabla_\disc c^{(m+1)})-\cA(g(\bar c^{(m+1)}),\nabla\bar c^{(m+1)}), \Pi_\disc u\rangle\\
&\quad\quad+\Big[C_3\delta t+S_\disc(\bar c(t^{(m+1)})) + S_\disc(\partial_t\bar c^{(m+1)}) + W_\disc(\nabla\bar c^{(m+1)}) \Big]
\Big\|\nabla_\disc u \Big\|_{L^2(\O)^d}.
\end{aligned}
\end{equation}
Let $u:=\delta t^{(m+\frac{1}{2})}\cE^{(m+1)}$ in \eqref{eq-proof-6-u}. Summing over $m = 0,...,k-1$ for some $k\in \{1,...,N\}$ gives
\begin{equation}\label{eq-proof-7-u}
\begin{aligned}
&\dsp\sum_{m=0}^{k-1} \langle \Pi_\disc \cE^{(m+1)}-\Pi_\disc \cE^{(m)},\Pi_\disc E^{(m+1)} \rangle
+\lambda\dsp\sum_{m=0}^{k-1} \delta t^{(m+\frac{1}{2})} \left\| \nabla_\disc \cE^{(m+1)} \right\|_{L^2(\O)^d}^2\\
&\leq\dsp\sum_{m=0}^{k-1} \delta t^{(m+\frac{1}{2})} \langle f(\bar c^{(m+1)})-f(\Pi_\disc c^{(m+1)}),\Pi_\disc \cE^{(m+1)}\rangle\\
&\quad+\dsp\sum_{m=0}^{k-1} \delta t^{(m+\frac{1}{2})}\langle \cA(g(\Pi_\disc c^{(m+1)}),\nabla_\disc c^{(m+1)})-\cA(g(\bar c^{(m+1)}),\nabla\bar c^{(m+1)}), \Pi_\disc \cE^{(m+1)}\rangle\\
&\quad+\dsp\sum_{m=0}^{k-1} \delta t^{(m+\frac{1}{2})}\left[C_3\delta t+S_\disc(\bar c(t^{(m+1)})) + S_\disc(\partial_t\bar c^{(m+1)}) + W_\disc(\nabla\bar c^{(m+1)}) \right]\left\| \nabla_\disc \cE^{(m+1)} \right\|_{L^2(\O)^d}.
\end{aligned}
\end{equation}
By applying the relation $a_1(a_1-a_2)\geq \frac{1}{2}a_1^2-\frac{1}{2}a_2^2$ to the first term on the left-hand side, we derive 
\[
\begin{aligned}
&\frac{1}{2}\langle \Pi_\disc \cE^{(k)},\Pi_\disc \cE^{(k)} \rangle
+\lambda\dsp\sum_{m=0}^{k-1} \delta t^{(m+\frac{1}{2})} \left\| \nabla_\disc \cE^{(m+1)} \right\|_{L^2(\O)^d}^2\\
&\leq \frac{1}{2} \langle \Pi_\disc \cE^{(0)},\Pi_\disc \cE^{(0)} \rangle
+\dsp\sum_{m=0}^{k-1} \delta t^{(m+\frac{1}{2})} \langle f(\bar c^{(m+1)})-f(\Pi_\disc c^{(m+1)}), \Pi_\disc \cE^{(m+1)}\rangle\\
&\quad+\dsp\sum_{m=0}^{k-1} \delta t^{(m+\frac{1}{2})}\langle \cA(g(\Pi_\disc c^{(m+1)}),\nabla_\disc c^{(m+1)})-\cA(g(\bar c^{(m+1)}),\nabla\bar c^{(m+1)}), \Pi_\disc \cE^{(m+1)}\rangle\\
&\quad+\dsp\sum_{m=0}^{k-1} \delta t^{(m+\frac{1}{2})}\left[C_3\delta t+S_\disc(\bar c(t^{(m+1)})) + S_\disc(\partial_t\bar c^{(m+1)}) + W_\disc(\nabla\bar c^{(m+1)}) \right]\left\| \nabla_\disc \cE^{(m+1)} \right\|_{L^2(\O)^d}.
\end{aligned}
\]
Utilising Young's inequality with a small parameter $\varepsilon_1>0$, we can estimate the last term on the right-hand side. It yields
\begin{equation}\label{eq-proof-8-u}
\begin{aligned}
&\frac{1}{2}\langle \Pi_\disc \cE^{(k)},\Pi_\disc \cE^{(k)} \rangle
+\lambda\dsp\sum_{m=0}^{k-1} \delta t^{(m+\frac{1}{2})} \left\| \nabla_\disc \cE^{(m+1)} \right\|_{L^2(\O)^d}^2 \\
&\leq \frac{1}{2} \langle \Pi_\disc \cE^{(0)},\Pi_\disc \cE^{(0)} \rangle
+\dsp\sum_{m=0}^{k-1} \delta t^{(m+\frac{1}{2})} \langle f(\bar c^{(m+1)})-f(\Pi_\disc c^{(m+1)}), \Pi_\disc \cE^{(m+1)}\rangle\\
&\quad+\dsp\sum_{m=0}^{k-1} \delta t^{(m+\frac{1}{2})}\langle \cA(g(\Pi_\disc c^{(m+1)}),\nabla_\disc c^{(m+1)})-\cA(g(\bar c^{(m+1)}),\nabla\bar c^{(m+1)}), \Pi_\disc \cE^{(m+1)}\rangle\\
&\quad+\frac{1}{2\varepsilon_1}\dsp\sum_{m=0}^{k-1} \delta t^{(m+\frac{1}{2})}\left[ C_3\delta t+S_\disc(\bar c(t^{(m+1)})) + S_\disc(\partial_t\bar c^{(m+1)}) + W_\disc(\nabla\bar c^{(m+1)}) \right]^2\\
&\quad+\frac{\varepsilon_1}{2}\dsp\sum_{m=0}^{k-1} \delta t^{(m+\frac{1}{2})} \left\| \nabla_\disc \cE^{(m+1)} \right\|_{L^2(\O)^d}^2.
\end{aligned}
\end{equation}
Let us now focus on the right-hand side of the above inequality. Firstly, we note that
\begin{equation}\label{eq-proof-9-u}
\begin{aligned}
\| \Pi_\disc \cE^{(0)} \|_{L^2(\O)} &\leq \| \Pi_\disc P_\disc \bar c(0)-\bar c(0) \|_{L^2(\O)}
+\| \bar c(0)-\Pi_\disc J_\disc\bar c(0) \|_{L^2(\O)}\\
&\leq S_\disc(\bar c(0)) + \| \bar c(0)-\Pi_\disc J_\disc\bar c(0) \|_{L^2(\O)}.
\end{aligned}
\end{equation}
In what follows, we let $\cE_\disc^0:=\| \bar c(0)-\Pi_\disc J_\disc\bar c(0) \|_{L^2(\O)}$. Utilising the Cauchy–Schwarz inequality and the assumption of Lipschitz continuity on $F$, one can express
\begin{equation}\label{eq-proof-F}
\begin{aligned}
&\dsp\sum_{m=0}^{k-1} \delta t^{(m+\frac{1}{2})} \langle f(\bar c^{(m+1)})-f(\Pi_\disc c^{(m+1)}), \Pi_\disc \cE^{(m+1)}\rangle\\
&\leq \ell_3 \dsp\sum_{m=0}^{k-1} \delta t^{(m+\frac{1}{2})} \Big\|\bar c^{(m+1)}-\Pi_\disc c^{(m+1)}\Big\|_{L^2(\O)}
\|\Pi_\disc \cE^{(n+1)}\|_{L^2(\O)}\\
&\leq \ell_3 \dsp\sum_{m=0}^{k-1} \delta t^{(m+\frac{1}{2})}
\Big[ \Big\|\Pi_\disc \cE^{(m+1)}\Big\|_{L^2(\O)} \Big\|\bar c^{(m+1)}-\Pi_\disc P_\disc \bar c(t^{(m+1)})\Big\|_{L^2(\O)}\\
&\quad+\Big\|\Pi_\disc \cE^{(m+1)}\Big\|_{L^2(\O)}^2 \Big].
\end{aligned}
\end{equation}
By employing the triangle inequality, along with the definition \eqref{PD} of $P_\disc$ to $w:=\bar c(t^{(n+1)})$, we deduce
\begin{equation}\label{eq-proof-Pi-u}
\begin{aligned}
\Big\| \bar c^{(m+1)}&-\Pi_\disc  P_\disc \bar c(t^{(m+1)}) \Big\|_{L^2(\O)}
\\
&\leq \Big\| \bar c^{(m+1)}-\bar c(t^{(m+1)}) \Big\|_{L^2(\O)} + S_\disc(\bar c(t^{(m+1)}))\\
&\leq C_4\delta t+S_\disc(\bar c(t^{(m+1)})).
\end{aligned}
\end{equation}
Plugging this estimate into \eqref{eq-proof-F} and applying Young's inequality with a small parameter $\varepsilon_2>0$, we infer
\begin{equation}\label{eq-proof-FF-1}
\begin{aligned}
&\dsp\sum_{m=0}^{k-1} \delta t^{(m+\frac{1}{2})} \langle f(\bar c^{(m+1)})-f(\Pi_\disc c^{(m+1)}), \Pi_\disc \cE^{(m+1)}\rangle\\
&\leq \frac{\ell_3(1+2\varepsilon_2)}{2\varepsilon_2}\dsp\sum_{m=0}^{k-1} \delta t^{(m+\frac{1}{2})} \Big\|\Pi_\disc \cE^{(m+1)}\Big\|_{L^2(\O)}^2\\
&\quad+\frac{\ell_3 \varepsilon_2}{2}\dsp\sum_{m=0}^{k-1} \delta t^{(m+\frac{1}{2})}\Big(C_4\delta t+S_\disc(\bar c(t^{(m+1)}))\Big)^2.
\end{aligned}
\end{equation}
Let us now turn to the fourth term in the right-hand side of the inequality \eqref{eq-proof-8-u}. We apply again the Cauchy–Schwarz inequality and the discrete Poincar\'e inequality \eqref{ponc-enq} in order to conclude the following formulation, thanks to the Lipschitz continuity assumptions on the operator $\cA$ and the function $g$
\begin{equation}\label{eq-proof-conv-1}
\begin{aligned}
&\dsp\sum_{m=0}^{k-1} \delta t^{(m+\frac{1}{2})}\langle \cA(g(\Pi_\disc c^{(m+1)}),\nabla_\disc c^{(m+1)})-\cA(g(\bar c^{(m+1)}),\nabla\bar c^{(m+1)}), \Pi_\disc \cE^{(m+1)}\rangle\\
&\leq \dsp\sum_{m=0}^{k-1}\delta t^{(m+\frac{1}{2})}\ell_1 \Big\|\Pi_\disc \cE^{(m+1)}\Big\|_{L^2(\O)}
\Big[
\Big\|g(\Pi_\disc c^{(m+1)}) - g(\bar c^{(m+1)})\Big\|_{L^2(\O)}\\
&\quad\quad+ \Big\|\nabla_\disc c^{(m+1)} - \nabla\bar c^{(m+1)} \Big\|_{L^2(\O)^d}
\Big]
\\
&\leq \dsp\sum_{m=0}^{k-1}\delta t^{(m+\frac{1}{2})}\Big[
\ell_1 \ell_2 \Big\|\Pi_\disc c^{(m+1)} - \bar c^{(m+1)}\Big\|_{L^2(\O)} \Big\|\Pi_\disc \cE^{(m+1)}\Big\|_{L^2(\O)}\\
&\quad+ \ell_1 \Big\|\nabla_\disc c^{(m+1)} - \nabla\bar c^{(m+1)} \Big\|_{L^2(\O)^d} \Big\|\Pi_\disc \cE^{(m+1)}\Big\|_{L^2(\O)}
\Big]\\
&\leq \dsp\sum_{m=0}^{k-1}\delta t^{(m+\frac{1}{2})}\Big[
\ell_1\ell_2 \Big\|\Pi_\disc \cE^{(m+1)}\Big\|_{L^2(\O)}^2
+\ell_1\ell_2 \Big\|\Pi_\disc P_\disc \bar c(t^{(m+1)}) - \bar c^{(m+1)}\Big\|_{L^2(\O)} \Big\|\Pi_\disc \cE^{(m+1)}\Big\|_{L^2(\O)} \\
&\quad+\ell_2 C_\disc \Big\|\nabla_\disc \cE^{(m+1)}\Big\|_{L^2(\O)^d}^2
+\ell_2 C_\disc \Big\|\nabla_\disc P_\disc \bar c(t^{(m+1)}) - \nabla\bar c^{(m+1)}\Big\|_{L^2(\O)^d} \Big\|\nabla_\disc \cE^{(m+1)}\Big\|_{L^2(\O)^d}
\Big].
\end{aligned}
\end{equation}
Employing Young's inequalities with small parameters $\varepsilon_3,\varepsilon_4>0$, we obtain  
\begin{equation}\label{eq-proof-conv-2}
\begin{aligned}
&\dsp\sum_{m=0}^{k-1} \delta t^{(m+\frac{1}{2})}\langle \cA(g(\Pi_\disc c^{(m+1)}),\nabla_\disc c^{(m+1)})-\cA(g(\bar c^{(m+1)}),\nabla\bar c^{(m+1)}), \Pi_\disc \cE^{(m+1)}\rangle\\
&\leq \dsp\sum_{m=0}^{k-1} \Big[
\frac{\ell_1^2 \ell_2^2 \varepsilon_2}{2}\Big\|\Pi_\disc P_\disc \bar c(t^{(m+1)}) - \bar c^{(m+1)}\Big\|_{L^2(\O)}^2 
+\frac{\ell_1^2 C_\disc^2 \varepsilon_2}{2}\Big\|\nabla_\disc P_\disc \bar c(t^{(m+1)}) - \nabla\bar c^{(m+1)} \Big\|_{L^2(\O)^d}^2\\
&\quad+\frac{1+2\varepsilon_2\ell_1\ell_2}{2\varepsilon_2}\Big\|\Pi_\disc \cE^{(m+1)}\Big\|_{L^2(\O)}^2
+\frac{1+2\varepsilon_3 C_\disc\ell_2}{2\varepsilon_2}\Big\|\nabla_\disc \cE^{(m+1)}\Big\|_{L^2(\O)^d}^2
\Big].
\end{aligned}
\end{equation}
As in \eqref{eq-proof-Pi-u}, we can estimate the second and third terms on the right-hand side as follows
\[
\begin{aligned}
&\|\Pi_\disc P_\disc \bar c(t^{(m+1)}) - \bar c^{(m+1)}\|_{L^2(\O)}^2 
+\|\nabla_\disc P_\disc \bar c(t^{(m+1)}) - \nabla\bar c^{(m+1)} \|_{L^2(\O)^d}^2\\
&\leq C_4\delta t+S_\disc(\bar c(t^{(m+1)}))
+C_3\delta t+S_\disc(\bar c(t^{(m+1)})).
\end{aligned}
\]
Together with \eqref{eq-proof-conv-2}, this yields
\begin{equation}\label{eq-proof-conv-3}
\begin{aligned}
&\dsp\sum_{m=0}^{k-1} \delta t^{(m+\frac{1}{2})}\langle \cA(g(\Pi_\disc c^{(m+1)}),\nabla_\disc c^{(m+1)})-\cA(g(\bar c^{(m+1)}),\nabla\bar c^{(m+1)}), \Pi_\disc \cE^{(m+1)}\rangle\\
&\leq \dsp\sum_{m=0}^{k-1}\delta t^{(m+\frac{1}{2})} \Big[
\Big( \frac{\ell_1^2 \ell_2^2 \varepsilon_2C_4 + \ell_1^2 C_\disc^2 \varepsilon_2C_3 }{2}\delta t +\frac{\ell_1^2 \ell_2^2 \varepsilon_2 + \ell_1^2 C_\disc^2 \varepsilon_2 }{2} S_\disc(\bar c(t^{(m+1)})) \Big)^2\\
&\quad+\frac{1+2\varepsilon_2\ell_1\ell_2}{2\varepsilon_2}\Big\|\Pi_\disc \cE^{(m+1)}\Big\|_{L^2(\O)}^2
+ \frac{1+2\varepsilon_3 C_\disc\ell_2}{2\varepsilon_2}\Big\|\nabla_\disc \cE^{(m+1)}\Big\|_{L^2(\O)^d}^2
\Big].
\end{aligned}
\end{equation}
By substituting \eqref{eq-proof-9-u}, \eqref{eq-proof-FF-1}, and \eqref{eq-proof-conv-3} into \eqref{eq-proof-8-u}, and using $\sum_{n=0}^{m-1}\delta t^{(n+\frac{1}{2})}\leq T$, we infer
\begin{equation}\label{new-E-u}
\begin{aligned}
&\frac{1}{2}\langle \Pi_\disc  \cE^{(k)}, \Pi_\disc  \cE^{(k)} \rangle
+(\lambda-\frac{1+2\varepsilon_3 C_\disc\ell_2}{2\varepsilon_2})\dsp\sum_{m=0}^{k-1} \delta t^{(m+\frac{1}{2})} \left\| \nabla_\disc  \cE^{(m+1)} \right\|_{L^2(\O)^d}^2 \\
&\leq\frac{(1+2\varepsilon_2)\ell_3 + 2\varepsilon_2\ell_1\ell_2 + 1}{2\varepsilon_2} \dsp\sum_{m=0}^{k-1} \delta t^{(m+\frac{1}{2})}\Big\|\Pi_\disc  \cE^{(m+1)}\Big\|_{L^2(\O)}^2\\
&\quad
+C_5\dsp\sum_{m=0}^{k-1} \delta t^{(m+\frac{1}{2})}(\mathbb M_\disc^{(m+1)})^2
+\Big(S_\disc(\bar c(0))+ \cE_\disc^0\Big)^2,
\end{aligned}
\end{equation}
where $\mathbb M_\disc^{(m+1)}$ is defined by \eqref{eq-MD} and $C_5$ depends on $C_3$ and $C_4$. The direct application of the discrete Gronwall’s Lemma \cite[Lemma 10.5]{B2} to the inequality \eqref{new-E-u} yields
\begin{equation}\label{eq-proof-10-u}
\begin{aligned}
&\frac{1}{2}\Big\| \Pi_\disc \cE^{(k)} \Big\|_{L^2(\O)}^2
+(\lambda-\frac{1+2\varepsilon_3 C_\disc\ell_2}{2\varepsilon_2})\dsp\sum_{m=0}^{k-1} \delta t^{(m+\frac{1}{2})} \left\| \nabla_\disc  E^{(m+1)} \right\|_{L^2(\O)^d}^2 \\
&\leq
\exp\Big\{\frac{T((1+2\varepsilon_2)\ell_3 + 2\varepsilon_2\ell_1\ell_2 + 1)}{2\varepsilon_2}\Big\}\dsp\sum_{m=0}^{k-1} \delta t^{(m+\frac{1}{2})}C_5(\mathbb M_\disc^{(m+1)})^2\\
&\quad +\Big(S_\disc(\bar c(0))+ \cE_\disc^0\Big)^2.
\end{aligned}
\end{equation}
From the triangle inequality, \eqref{new-eq-proof-100}, and \eqref{eq-proof-10-u} combined with the power-of-sums inequality $(a_1+a_2)^{1/2}\leq a_1^{1/2}+a_2^{1/2}$, we obtain 
\begin{equation}\label{eq-proof-11-u}
\begin{aligned}
&\|\Pi_\disc c^{(k)}-\bar c(t^{(k)})\|_{L^2(\O)}\\
&\leq \|\Pi_\disc \cE^{(k)}\|_{L^2(\O)}
+\|\Pi_\disc P_\disc\bar c(t^{(k)})-\bar c(t^{(k)})\|_{L^2(\O)}\\
&\leq 
C_6\Big[\dsp\sum_{m=0}^{k-1} \delta t^{(m+\frac{1}{2})}\mathbb M_\disc^{(m+1)}
+S_\disc(\bar c(0))+ \cE_\disc^0\Big]
+\sqrt{2}S_\disc(\bar c(t^{(k)})), \quad \forall k\in\{1,...,N\},
\end{aligned}
\end{equation}
and  
\begin{equation}\label{eq-proof-13-u}
\begin{aligned}
&\dsp\sum_{m=0}^{N-1}\delta t^{(m+\frac{1}{2})}\Big\|\nabla_\disc c^{(m+1)} - \nabla\bar c(t^{(m+1)})\Big\|_{L^2(\O)^d}\\
&\leq \dsp\sum_{m=0}^{N-1}\delta t^{(m+\frac{1}{2})}\Big\|\nabla_\disc \cE^{(m+1)}\Big\|_{L^2(\O)^d}
+\dsp\sum_{m=0}^{N-1}\delta t^{(m+\frac{1}{2})}\Big\|\nabla_\disc P_\disc\bar c(t^{(m+1)})-\nabla\bar c(t^{(m+1)})\Big\|_{L^2(\O)^d}\\
&\leq 
C_7\Big[\dsp\sum_{m=0}^{k-1} \delta t^{(m+\frac{1}{2})}\mathbb M_\disc^{(m+1)}
+S_\disc(\bar c(0))+ \cE_\disc^0\Big]
+\sqrt{2}\dsp\sum_{m=0}^{k-1} \delta t^{(m+\frac{1}{2})}S_\disc(\bar c(t^{(m+1)})).
\end{aligned}
\end{equation}
Using the triangle inequality and the estimates \eqref{eq-proof-11-u} and \eqref{eq-proof-13-u} implies the desired estimates \eqref{eq-error-1} and \eqref{eq-error-2}, thanks to the Lipschitz-continuity of $\bar c,\nabla\bar c:[0,T]\to H^1(\O)$ to control the quantities $\bar c(\cdot,t)-\bar c(t^{(m+1)})$ and $\nabla\bar c(\cdot,t)-\nabla\bar c(t^{(m+1)})$ for any $t\in (t^{(n)},t^{(n+1)}]$.
\end{proof}	

\begin{remark}
Theorem \ref{thm-err-rm} establishes the rate of convergence based on the mesh size $h$ and the time discretisation. \cite[Remark 2.24]{B1} demonstrates the connection between mesh size and the parameters $C_\disc$, $S_\disc$ and $W_\disc$ for mesh-based gradient discretisation, as follows:
\[
S_\disc(w)\leq h\|w\|_{H^2(\O)}, \quad \mbox{ for all } w \in H^2(\O),
\]
\[
W_\disc(\bxi)\leq h\|\bxi\|_{H^1(\O)^d}, \quad \mbox{ for all } \bxi \in H^1(\O)^d.
\]

\end{remark}


\section{Numerical results}\label{sec-numerical}
We consider here the generalised Burgers-Fisher (GBF) equation as an application of the time-dependent convection-diffusion-reaction model \eqref{eq-conv-1}--\eqref{eq-conv-3}, with setting
\[
\cA(u,\bphi):=\dsp\sum_{i=1}^d u\varphi_i,\quad \forall u\in L^2(0,T;L^2(\O)), \bphi=(\varphi_1,...,\varphi_d) \in L^2(0,T;L^2(\O)^d),
\]
\[
g(\bar c)= \bar c^p,\; f(\bar c)=\bar c(1-\bar c^p), \mbox{ and } \lambda=1.
\] 
where $p$ is a positive constant, over the domain $\Omega=[0,1]^2$. The exact solution for such a model is given by \cite{N1,N2,N3}
 \begin{equation}\label{eq-GBF-exact}
              \bar c(x,y,t)=\left[\frac{1}{2}+\frac{1}{2}\tanh\left(\frac{-2p}{4(p+1)}\left(x+y-(\frac{4+2(p+1)^2}{2(p+1)})t\right)\right)\right]^\frac{1}{p},
\end{equation}
 where the initial and the Dirichlet boundary conditions are extracted accordingly.
To assess the validity of the gradient scheme, we discretise the GBF equation by the scheme \eqref{eq-gs} with the gradient discretisation corresponding to the hybrid mimetic mixed (HMM) method (a unified framework that combines three distinct schemes: the hybrid finite volume method, the mixed finite volume method, and the (mixed-hybrid) mimetic finite difference method. \cite[Chapter 13]{B1}. Notably, it can be applied on general meshes without requiring orthogonality assumptions. For the sake of completeness we briefly recall the definition of this gradient discretisation. Let $\mathcal T=(\mesh,\mathcal F,\mathcal V)$ be the polytopal mesh of the spatial domain $\O$ described in \cite[Definition 7.2]{B1}, where  $\mesh$ is the set of polygonal cells $K$, $\cF$ is the set of edges $\edge$, and $\mathcal V$ is a set of points $(\x_K)_{K\in\mesh}$. The discrete space $X_{\disc,0}$, and operators $\Pi_\disc$ and $\nabla_\disc$ are given by:

\begin{itemize}
\item $
X_{\disc,0}=\{ \varphi=((\varphi_{K})_{K\in \mesh}, (\varphi_{\sigma})_{\sigma \in \cF})\;:\; \varphi_{K},\, \varphi_{\sigma} \in \RR \mbox{ and }\varphi_\edge=0,\; \forall \edge \in \cF \cap \dr\O \}
$,
\item $\Pi_\disc \varphi=\varphi_K$ on $K$, for all $\varphi\in X_{\disc,0}$ and all $K\in\mesh$, and for a.e. $\x\in K$,
\item $\nabla_\disc \varphi=\dsp\frac{1}{|K|}\dsp\sum_{\edge\in\cF_K}|\edge|\bfn_{K,\edge}+S_{K,\edge}(\varphi)$, for all $\varphi\in X_{\disc,0}$, all $K\in\mesh$, and all $\edge\in\cF_K$ ( the
set of edges of $K$), where $S_{K,\edge}$ is a stabilisation term depending on a cell $K$ and its edges, and $\bfn_{K,\edge}$ is is the unit vector normal to $\edge$ outward to cell $K$.
\end{itemize}

For computational purpose, we can rewrite the HMM method as a conservative discretisation (the mixed finite volume method). To do so, we consider the following linear flux 
\begin{align*}
\sum_{\sigma \in \cF_K}|\sigma| F_{K,\sigma}(\varphi)
(v_K-v_\sigma)={}&\int_K \nabla_\disc \varphi\cdot\nabla_\disc v \ud \x, \mbox{ for all $K \in \mesh$ and all $\varphi,v\in X_{\disc,0}$}.
\end{align*} 

Therefore, the mixed finite volume method of the generalised Burgers-Fisher equation is, for all $K\in\mesh$ and for all $m=0,...,N-1$
\begin{subequations}\label{mont-obs}
\[
\begin{aligned}
&\frac{|K|}{\delta t^{(n+\frac{1}{2})}}\left( c_K^{(m+1)}-c_K^{(m)} \right)
+\sum_{\sigma \in \cF_K} F_{K,\sigma}(c^{(m+1)})\\
&+({c_K^{(m+1)}}^p, {c_K^{(m+1)}}^p)\cdot \frac{1}{|K|}\dsp\sum_{\edge\in\cF_K}|\edge|\bfn_{K,\edge}c_\edge^{(m+1)}
=c_K^{(m+1)}\left( 1-{c_K^{(m+1)}}^p \right),
\end{aligned}
\]
\[
F_{K,\edge}(c^{(m+1)})+F_{L,\edge}(c^{(m+1)})=0, \quad \mbox{ for all } \sigma \in \cF_K \cap \cF_L, K\neq L,
\]
\[
c_\sigma^{(m+1)}=0, \quad \mbox{ for all } \sigma \in \cF \cap \dr\O,
\]
\[
c^{(0)}=c_{\rm ini}(\x_K,0), \quad \mbox{ for all } K\in \mesh.
\]
\end{subequations}

A key advantage of the HMM method is its flexibility in handling different types of meshes across multiple spatial dimensions, with minimal restrictions on control volumes. To demonstrate its effectiveness, we examine two examples of generalised GBF model using the HMM method with $p=2$ and $p=0.5$. In both cases, the method is tested on four distinct types of general meshes introduced in \cite{HH08}. Examples of these meshes, each with varying cell counts, are shown in Figure \ref{fig-mesh}. Mesh type (b) consists primarily of hexagonal cells, while type (c) is a conforming distorted quadrangular mesh. Unlike the triangular mesh (type (a)), both types (b) and (c) are skewed in different directions across the domain and contain elements of various shapes. Type (d) is a locally refined, non-conforming rectangular mesh with a refinement concentrated in the lower-left corner.  As demonstrated in \cite{HH08}, this selection of mesh types is representative of generic meshes commonly encountered in practical applications. 
Simulations are conducted for all test cases up to $T=1$, using time steps $\delta t=0.01$, $0.005$, $0.00025$, and $0.00125$ corresponding to increasingly finer meshes. 

The resulting errors and the corresponding orders of convergence with respect to the mesh size $h$ are presented in Tables \ref{tab1}-\ref{tab8} for both test cases. Results for the triangular mesh are given in Tables \ref{tab1} and \ref{tab5}, for the hexagonal mesh in Tables \ref{tab2} and \ref{tab6}, for the conforming distorted quadrangular mesh in Tables \ref{tab3} and \ref{tab7}, and in Tables \ref{tab4} and \ref{tab8} for the locally refined non-conforming rectangular mesh.
 Our findings indicate that, for the triangular, hexagonal, and locally refined non-conforming rectangular meshes, the convergence rates of the $L_2$ relative errors on $\bar c$ and $\nabla\bar c$ are approximately $1$. These results align with expected behaviour for lower-order methods, such as the HMM. Interestingly, for the conforming distorted quadrangular mesh (Tables \ref{tab3} and  \ref{tab7}), the order of convergence for $L_2$ relative errors on $\bar c$ and $\nabla\bar c$ exceeds $1$, indicating a potential super-convergence property of the HMM in this case.
\begin{figure}[ht]
	\begin{center}
	\begin{tabular}{cc}
	\includegraphics[width=0.40\linewidth]{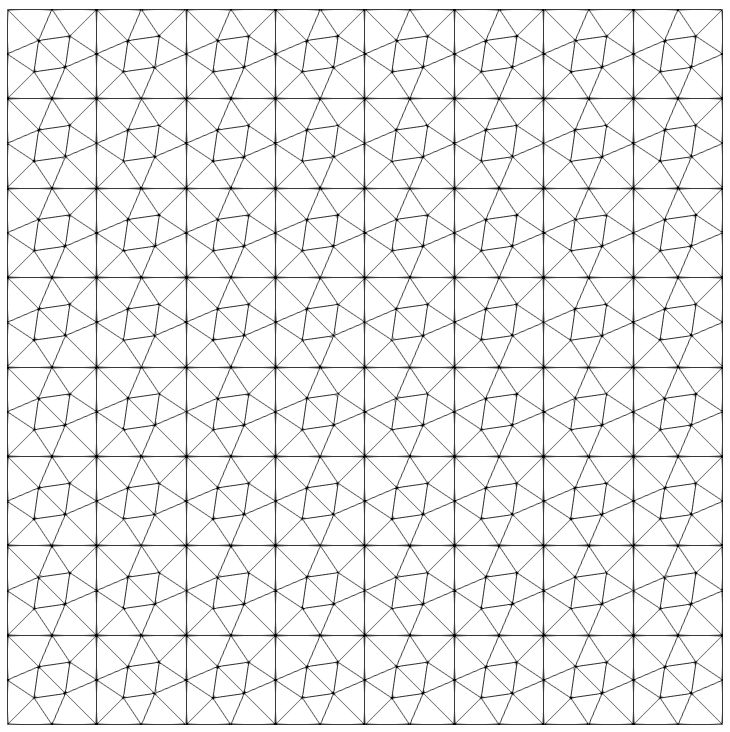}&
	\includegraphics[width=0.40\linewidth]{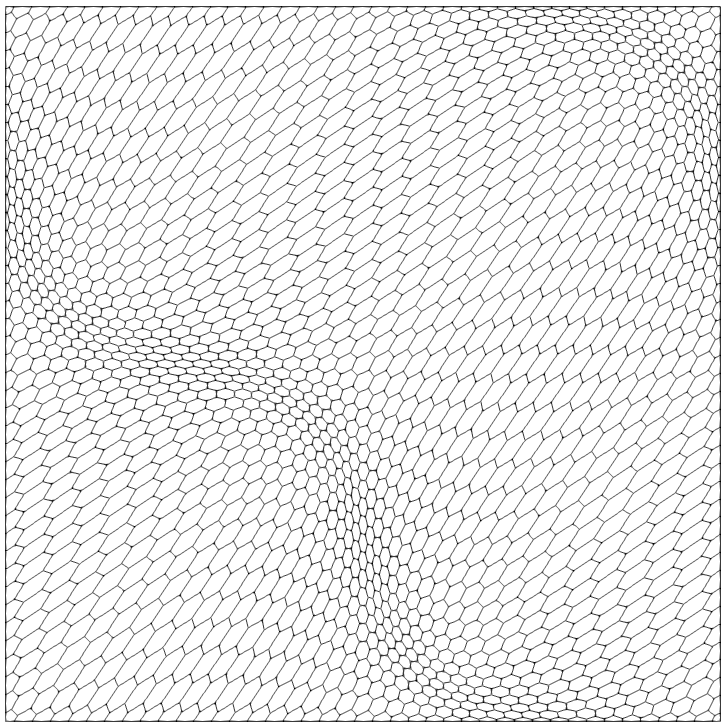}\\
	\texttt{(a) triangular mesh}&\texttt{(b) Hexagonal mesh}\\
	\includegraphics[width=0.40\linewidth]{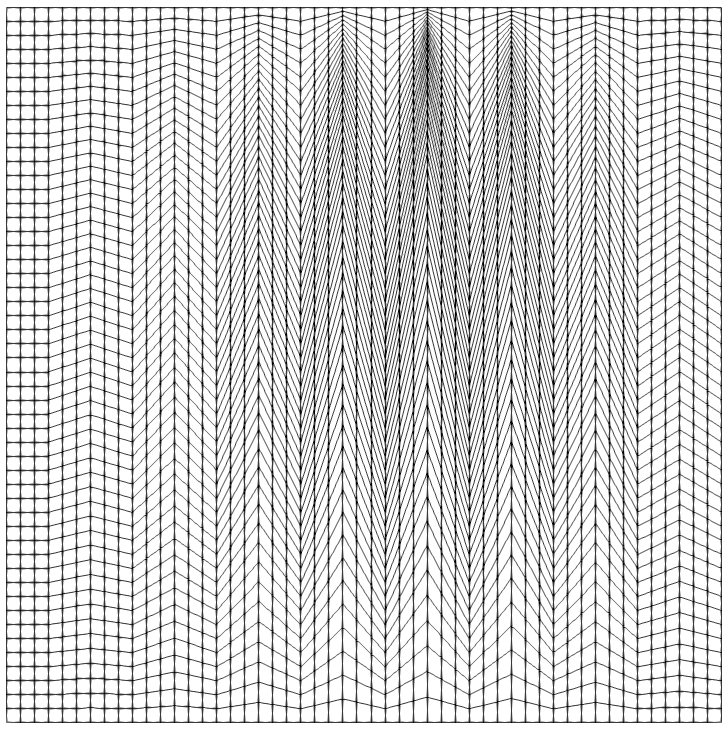}&
	\includegraphics[width=0.40\linewidth]{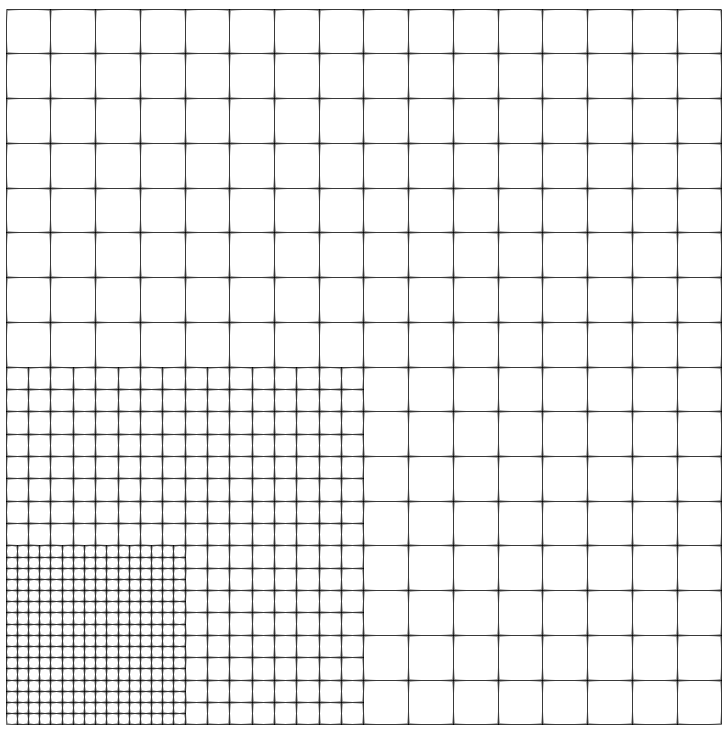}\\
	\texttt{(c) Distroted mesh }&\texttt{(d) Locally refined and non conforming mesh }\\
	\end{tabular}
	\end{center}
	\caption{Samples of the various 2D meshes}
	\label{fig-mesh}
\end{figure}
\begin{table}[]
\begin{tabular}{c c c c c}
\hline
$h$ &
$\frac{\| \bar c(\cdot,T) - \Pi_\disc c^N\|_{L^{2}(\O)}}{\|\bar c(\cdot,T)\|_{L^2(\O)}}$&
rate&
$\frac{\| \nabla \bar c(\cdot,T) - \nabla_\disc c^N\|_{L^{2}(\O)^2}}{\|\nabla\bar c(\cdot,T)\|_{L^2(\O)^2}}$&
rate 
\\ \hline
0.1250000&
0.0000441&
--&
0.0115277&
--
\\ 
0.0625000&
0.0000183&
1.2661441&
0.0057652&
0.9996583
\\ 
0.0312500&
0.0000083&
1.1460176&
0.0028830&
0.9997962
\\ 
0.01562500&
0.00000393&
1.07651504&
0.00144161&
0.99988925
\\ \hline
\end{tabular}
\caption{The relative errors on $\bar c$ and $\nabla\bar c$ and the convergence rates w.r.t.  $h$ the size of triangular mesh with $p=2$.}
\label{tab1}   
\end{table}
\begin{table}[]
\begin{tabular}{c c c c c}
\hline
$h$ &
$\frac{\| \bar c(\cdot,T) - \Pi_\disc c^N\|_{L^{2}(\O)}}{\|\bar c(\cdot,T)\|_{L^2(\O)}}$&
rate&
$\frac{\| \nabla \bar c(\cdot,T) - \nabla_\disc c^N\|_{L^{2}(\O)^2}}{\|\nabla\bar c(\cdot,T)\|_{L^2(\O)^2}}$&
rate 
\\ \hline   
0.1297130&
0.0000374&
--&
0.0026991&
--
\\ 
0.0657364&
0.0000168&
1.1779549&
0.0011768&
1.2213457
\\ 
0.0329800&
0.0000079&
1.0922849&
0.0005337&
1.1462540
\\ 
0.0165040&
0.0000038&
1.0454178&
0.0002529&
1.0790024
\\ \hline
\end{tabular}
\caption{The relative errors on $\bar c$ and $\nabla\bar c$ and the convergence rates w.r.t. $h$ the size of hexagonal mesh with $p=2$.}
\label{tab2}   
\end{table}
\begin{table}[]
\begin{tabular}{c c c c c}
\hline
$h$ &
$\frac{\| \bar c(\cdot,T) - \Pi_\disc c^N\|_{L^{2}(\O)}}{\|\bar c(\cdot,T)\|_{L^2(\O)}}$&
rate&
$\frac{\| \nabla \bar c(\cdot,T) - \nabla_\disc c^N\|_{L^{2}(\O)^2}}{\|\nabla\bar c(\cdot,T)\|_{L^2(\O)^2}}$&
rate 
\\ \hline
0.1665956&
0.0000225&
--&
0.0021771&
--
\\ 
0.1115566&
0.0000115&
1.6737405&
0.0010651&
1.7827569
\\ 
0.0838522&
0.0000056&
2.5259207&
0.0005446&
2.3493490
\\ 
0.06717051&
0.00000265&
3.35803175&
0.00029008&
2.83994585
\\ \hline
\end{tabular}
\caption{The relative errors on $\bar c$ and $\nabla\bar c$ and the convergence rates w.r.t.  $h$ the size of distorted mesh with $p=2$.}
\label{tab3}   
\end{table}
   \begin{table}[]
\begin{tabular}{c c c c c}
\hline
$h$ &
$\frac{\| \bar c(\cdot,T) - \Pi_\disc c^N\|_{L^{2}(\O)}}{\|\bar c(\cdot,T)\|_{L^2(\O)}}$&
rate&
$\frac{\| \nabla \bar c(\cdot,T) - \nabla_\disc c^N\|_{L^{2}(\O)^2}}{\|\nabla\bar c(\cdot,T)\|_{L^2(\O)^2}}$&
rate 
\\ \hline   
0.1767767&
0.0000763&
--&
0.0022502&
--
\\ 
0.0883883&
0.0000260&
1.5517510&
0.0010401&
1.1133085
\\ 
0.0441948&
0.0000101&
1.2115446&
0.0004992&
1.0590740
\\ 
0.0220971&
0.0000044&
1.0619295&
0.0002445&
1.0295922
\\ \hline
\end{tabular}
\caption{The relative errors on $\bar c$ and $\nabla\bar c$ and the convergence rates w.r.t. $h$ the size of a locally refined
non-conforming rectangular mesh with $p=2$.}
\label{tab4}   
\end{table}

\begin{table}[]
\begin{tabular}{c c c c c}
\hline
$h$ &
$\frac{\| \bar c(\cdot,T) - \Pi_\disc c^N\|_{L^{2}(\O)}}{\|\bar c(\cdot,T)\|_{L^2(\O)}}$&
rate&
$\frac{\| \nabla \bar c(\cdot,T) - \nabla_\disc c^N\|_{L^{2}(\O)^2}}{\|\nabla\bar c(\cdot,T)\|_{L^2(\O)^2}}$&
rate 
\\ \hline
0.1250000&
0.0000471&
--&
0.0010114&
--
\\ 
0.0625000&
0.0000223&
1.0813452&
0.0005022&
1.0100528
\\ 
0.0312500&
0.0000108&
1.0382011&
0.0002503&
1.0047728
\\ 
0.01562500&
0.00000536&
1.01845908&
0.00012493&
1.00232597
\\ \hline
\end{tabular}
\caption{The relative errors on $\bar c$ and $\nabla\bar c$ and the convergence rates w.r.t.  $h$ the size of triangular mesh with $p=0.5$.}
\label{tab5}   
\end{table}
\begin{table}[]
\begin{tabular}{c c c c c}
\hline
$h$ &
$\frac{\| \bar c(\cdot,T) - \Pi_\disc c^N\|_{L^{2}(\O)}}{\|\bar c(\cdot,T)\|_{L^2(\O)}}$&
rate&
$\frac{\| \nabla \bar c(\cdot,T) - \nabla_\disc c^N\|_{L^{2}(\O)^2}}{\|\nabla\bar c(\cdot,T)\|_{L^2(\O)^2}}$&
rate 
\\ \hline   
0.1297130&
0.0000468&
--&
0.0006397&
--
\\ 
0.0657364&
0.0000224&
1.0856872&
0.0003078&
1.0762592
\\ 
0.0329800&
0.0000109&
1.0431010&
0.0001504&
1.0381276
\\ 
0.0165040&
0.0000054&
1.0217026&
0.0000743&
1.0185651
\\ \hline
\end{tabular}
\caption{The relative errors on $\bar c$ and $\nabla\bar c$ and the convergence rates w.r.t. $h$ the size of hexagonal mesh with 
$p=0.5$.}
\label{tab6}   
\end{table}

\begin{table}[]
\begin{tabular}{c c c c c}
\hline
$h$ &
$\frac{\| \bar c(\cdot,T) - \Pi_\disc c^N\|_{L^{2}(\O)}}{\|\bar c(\cdot,T)\|_{L^2(\O)}}$&
rate&
$\frac{\| \nabla \bar c(\cdot,T) - \nabla_\disc c^N\|_{L^{2}(\O)^2}}{\|\nabla\bar c(\cdot,T)\|_{L^2(\O)^2}}$&
rate 
\\ \hline
0.1665956&
0.0000398&
--&
0.0003984&
--
\\ 
0.1115566&
0.0000200&
1.7141383&
0.0005947&
1.7361804
\\ 
0.0838522&
0.0000099&
2.4545544&
0.0001484&
2.4228989
\\ 
0.06717051&
0.00000489&
3.20261241&
0.00007470&
3.09592558
\\ \hline
\end{tabular}
\caption{The relative errors on $\bar c$ and $\nabla\bar c$ and the convergence rates w.r.t.  $h$ the size of distorted mesh with $p=0.5$.}
\label{tab7}   
\end{table}
\begin{table}[]
\begin{tabular}{c c c c c}
\hline
$h$ &
$\frac{\| \bar c(\cdot,T) - \Pi_\disc c^N\|_{L^{2}(\O)}}{\|\bar c(\cdot,T)\|_{L^2(\O)}}$&
rate&
$\frac{\| \nabla \bar c(\cdot,T) - \nabla_\disc c^N\|_{L^{2}(\O)^2}}{\|\nabla\bar c(\cdot,T)\|_{L^2(\O)^2}}$&
rate 
\\ \hline   
0.1767767&
0.0000596&
--&
0.0006732&
--
\\ 
0.0883883&
0.0000251&
1.2469915&
0.0003152&
1.0949167
\\ 
0.0441948&
0.0000115&
1.1242784&
0.0001522&
1.0499478
\\ 
0.0220971&
0.0000552&
1.0619295&
0.0000748&
1.0256775
\\ \hline
\end{tabular}
\caption{The relative errors on $\bar c$ and $\nabla\bar c$ and the convergence rates w.r.t. $h$ the size of a locally refined
non-conforming rectangular mesh with $p=0.5$.}
\label{tab8}   
\end{table}

\section{Conclusion}
A time-dependent convection-diffusion-reaction model was analysed within the general framework of the GDM. We established a novel error estimate in appropriate discrete norms, providing a rigorous foundation for the convergence of numerical approximations. To validate our theoretical findings, we applied the HMM method to the GBF model as a prototype example, conducting numerical experiments on four distinct types of general meshes. The computed errors and corresponding convergence orders with respect to the mesh size were reported. Our results demonstrate that the method achieves first-order convergence, even on severely distorted meshes, aligning well with the theoretical convergence rates.



\bibliographystyle{siam}
\bibliography{Ref}

\begin{thebibliography}{10}

\bibitem{A16}
{\sc N.~Alinia and M.~Zarebnia}, {\em A numerical algorithm based on a new kind
  of tension {B-spline} function for solving {Burgers-Huxley} equation},
  Numerical Algorithms, 82 (2019), p.~1121–1142.

\bibitem{A1}
{\sc J.~Burgers}, {\em A mathematical model illustrating the theory of
  turbulence}, in Advances in Applied Mechanics, Elsevier, 1948, pp.~171--199.

\bibitem{A12}
{\sc V.~Chandraker, A.~Awasthi, and S.~Jayaraj}, {\em Numerical treatment of
  {Burger-Fisher} equation}, Procedia Technology, 25 (2016), p.~1217–1225.

\bibitem{N1}
{\sc H.~Chen and H.~Zhang}, {\em New multiple soliton solutions to the general
  {Burgers–Fisher} equation and the {Kuramoto–Sivashinsky} equation},
  Chaos, Solitons and Fractals, 19 (2004), pp.~71--76.

\bibitem{A22}
{\sc A.~Cherati and H.~Momeni}, {\em Vittual element method for numerical
  simulation of {Burgers-Fisher} equation on convex and non-convex meshes},
  Mathematics Interdisciplinary Research, 9 (2024), pp.~1--22.

\bibitem{A21}
{\sc P.~Chin}, {\em The analysis of the solution of the {Burgers–Huxley}
  equation using the {Galerkin} method}, Numerical Methods for Partial
  Differential Equations, 39 (2023), pp.~2787--2807.

\bibitem{A4}
{\sc M.~Darvishi, S.~Kheybari, and F.~Khani}, {\em Spectral collocation method
  and {Darvishi’s} preconditionings to solve the generalized
  {Burgers–Huxley} equation}, Communications in Nonlinear Science and
  Numerical Simulation, 13 (2008), pp.~2091--2103.

\bibitem{B4}
{\sc L.~Debnath}, {\em Nonlinear partial differential equations for scientists
  and engineers}, Birkhauser, Basel, 2012.

\bibitem{B1}
{\sc J.~Droniou, R.~Eymard, T.~Gallou\"et, C.~Guichard, and R.~Herbin}, {\em
  The gradient discretisation method}, Mathematics \& Applications, Springer,
  Heidelberg, 2018.

\bibitem{A10}
{\sc V.~Ervina, J.~Macias-Díazb, and J.~Ruiz-Ramirez}, {\em A positive and
  bounded finite element approximation of the generalized {Burgers–Huxley}
  equation}, Journal of Mathematical Analysis and Applications, 424 (2015),
  p.~1143–1160.

\bibitem{A2}
{\sc R.~A. FISHER}, {\em The wave of advance of advantageous genes}, Annals of
  Eugenics, 7 (1937), pp.~355--369.

\bibitem{A5}
{\sc A.~Golbabai and M.~Javidi}, {\em A spectral domain decomposition approach
  for the generalized {Burger’s Fisher} equation}, Chaos Solitons and
  Fractals, 39 (2009), p.~385–392.

\bibitem{HH08}
{\sc R.~Herbin and F.~Hubert}, {\em Benchmark on discretization schemes for
  anisotropic diffusion problems on general grids}, in Finite volumes for
  complex applications V, ISTE, London, 2008, pp.~659--692.

\bibitem{A23}
{\sc G.~hua Gao, B.~Ge, and Y.~Chen}, {\em Derivation and error analysis of
  three-level linearized difference schemes for solving the {Burgers-Fisher}
  equation}, Journal of Difference Equations and Applications, 31 (2025),
  pp.~115--153.

\bibitem{A17}
{\sc M.~Hussain and S.~Haq}, {\em Numerical solutions of strongly non-linear
  generalized {Burgers–Fisher} equation via meshfree spectral technique},
  International Journal of Computer Mathematics, 98 (2020), p.~1727–1748.

\bibitem{N2}
{\sc H.~Ismail, K.~Raslan, and A.~A. Rabboh}, {\em Adomian decomposition method
  for {Burgers–Huxley} and {Burgers–Fisher} equations}, Applied Mathematics
  and Computation, 159 (2004), pp.~291--301.

\bibitem{A3}
{\sc M.~Javidi}, {\em Spectral collocation method for the solution of the
  generalized {Burger Fisher} equation}, Applied Mathematics and Computation,
  174 (2006), pp.~345--352.

\bibitem{A18}
{\sc A.~Khan, M.~Mohan, and R.~Ruiz-Baier}, {\em Conforming, nonconforming and
  {DG} methods for the stationary generalized {Burgers-Huxley} equation},
  Journal of Scientific Computing, 88 (2021), pp.~1--26.

\bibitem{A8}
{\sc B.~Kumar, V.~Sangwan, S.~Murthy, and M.~Nigam}, {\em A numerical study of
  singularly perturbed generalized {Burgers–Huxley} equation using three-step
  {Taylor–Galerkin} method}, Computers and Mathematics with Applications, 62
  (2011), p.~776–786.

\bibitem{A13}
{\sc J.~Macias-Díazb and A.~Gonzalez}, {\em A convergent and dynamically
  consistent finite-difference method to approximate the positive and bounded
  solutions of the classical {Burgers–Fisher} equation}, Journal of
  Computational and Applied Mathematics, 318 (2017), pp.~604--615.

\bibitem{A19}
{\sc M.~Mohan and A.~Khan}, {\em On the generalized {Burgers-Huxley} equation:
  Existence, uniqueness, regularity, global attractors and numerical studies},
  Discrete and Continuous Dynamical Systems - B, 26 (2021), pp.~3943--3988.

\bibitem{A15}
{\sc M.~Namjoo, M.~Zeinadini, and S.~Zibaei}, {\em Nonstandard
  finite-difference scheme to approximate the generalized {Burgers-Fisher}
  equation}, Mathematical Methods in the Applied Sciences, 41 (2018),
  p.~8212–8228.

\bibitem{A6}
{\sc M.~Sari, G.~Gürarslan, and I.~Dag}, {\em A compact finite difference
  method for the solution of the generalized {Burgers–Fisher} equation},
  Numerical Methods for Partial Differential Equations, 26 (2010),
  p.~125–134.

\bibitem{A7}
{\sc M.~Sari, G.~Gürarslan, and A.~Zeytinoglu}, {\em High-order finite
  difference schemes for numerical solutions of the generalized
  {Burgers–Huxley} equation}, Numerical Methods for Partial Differential
  Equations, 27 (2011), p.~1313–1326.

\bibitem{B3}
{\sc J.~Satsuma, M.~Ablowitz, B.~Fuchssteiner, and M.~Kruskal}, {\em Topics in
  Soliton Theory and Exactly Solvable Nonlinear Equations}, World Scientific,
  Singapore, 1987.

\bibitem{A20}
{\sc Shallu and V.~Kukreja}, {\em An improvised collocation algorithm to solve
  generalized {Burgers–Huxley} equation}, Arabian Journal of Mathematics, 11
  (2022), p.~379–396.

\bibitem{A9}
{\sc M.~Tatari, B.~Sepehrian, and M.~Alibakhshi}, {\em New implementation of
  radial basis functions for solving {Burgers-Fisher} equation}, Numerical
  Methods for Partial Differential Equations, 28 (2012), p.~248–262.

\bibitem{B2}
{\sc V.~Thom{\'e}e}, {\em {Galerkin} finite element methods for parabolic
  problems}, Springer Science \& Business Media, 2007.

\bibitem{N3}
{\sc A.~Wazwaz}, {\em The tanh method for generalized forms of nonlinear heat
  conduction and {Burgers–Fisher} equations}, Applied Mathematics and
  Computation, 169 (2005), pp.~321--338.

\bibitem{A14}
{\sc O.~Yadav and R.~Jiwari}, {\em Finite element analysis and approximation of
  {Burgers-Fisher} equation}, Numerical Methods for Partial Differential
  Equations, 33 (2017), pp.~1652--1677.

\bibitem{A11}
{\sc I.~Çelik}, {\em Chebyshev wavelet collocation method for solving
  generalized {Burgers-Huxley} equation}, Mathematical Methods in the Applied
  Sciences, 39 (2016), p.~366–377.

\end{thebibliography}

\end{document}